\documentclass[12pt]{amsart}
\usepackage{amssymb}
\usepackage{epsfig,amsfonts,amssymb,amsthm}
\usepackage{amsmath}
\theoremstyle{plain}
\newtheorem{theorem}{Theorem}[section]

\newtheorem{corollary}[theorem]{Corollary}
\newtheorem{lemma}[theorem]{Lemma}

\theoremstyle{definition}

\theoremstyle{remark}
\newtheorem{rem}{Remark}[section]

\newcommand{\ba}{\begin{eqnarray}}
\newcommand{\be}{\begin{equation}}
\newcommand{\ea}{\end{eqnarray}}
\newcommand{\ee}{\end{equation}}
\newcommand{\benn}{\begin{equation*}}
\newcommand{\eenn}{\end{equation*}}

\def\ve{\varepsilon}

\providecommand{\norm}[1]{\pmb{\lVert}#1\pmb{\rVert}}
\providecommand{\normu}[1]{\pmb{|}#1\pmb{|}}
\providecommand{\normf}[1]{\pmb{|}#1\pmb{|}_{*}}

\numberwithin{equation}{section}

\title{Estimate for a solution to the  water wave problem in the presence
of a submerged body}
\author{I. Kamotski, V. Maz'ya}
\address  {Department of Mathematics,
University College London, Gower Street, London, WC1E 6BT
}
\email{i.kamotski@ucl.ac.uk}
\address{Department of Mathematics, Link\"{o}ping University,
  SE-581 83 Link\"{o}ping, Sweden ; }
\address{Department of Mathematical Sciences, University of Liverpool, Liverpool, L69 7ZL, UK;}
\email{vlmaz@mai.liu.se}
\begin{document}

\bigskip

\begin{abstract}
We study the two-dimensional problem of propagation of linear water waves in deep water in the presence  of a  submerged body.
Under some geometrical requirements, we derive an explicit bound for the solution depending on the domain and the functions on the right-hand side.
\end{abstract}


\dedicatory{Dedicated to the memory of  Mark  Vishik}

\maketitle

\newpage
\section{\bf{Introduction}.}
\bigskip

\subsection{Boundary value problem}
We study the linear water waves problem for a two-dimensional
domain $\Omega$, which represents  water of infinite depth in the
presence of a  submerged body $B$.

We fix  a Cartesian system
$x=(x_1,x_2)$ with the origin $O$ and consider a bounded domain $B$,
$$\overline{B}\subset
\mathbb{R}^2_+=\{(x_1,x_2)\in\mathbb{R}^2:x_2>0\},$$ with smooth boundary $S$.
The role of the water
  surface is played by the line $\Gamma=\{x: x_2=0\}$ (note that that the axis $x_2$ is directed into water).

 Let the velocity potential $u$ be a smooth function in $\overline{\Omega}$ subject to the  equations:
\be
    -\Delta u=f \ \text{in}\  \Omega, \label{w1}
\ee \be
    \partial_n u=g_1 \ \text{on}\   S, \label{w2}
\ee \be
    \partial_n u-\nu u=g_2\ \text{on}\    \Gamma, \label{w3}
\ee where $n=(n_1,n_2)$ is the external normal to $\Omega$, $\nu$ is a positive spectral parameter and $f,g_1,g_2$ are  given
smooth functions, where   $f$ and  $g_2$ vanish at infinity with an appropriate rate of decay.

We are looking for   solutions which satisfy the
  {\it radiation conditions at  infinity} (see \eqref{radcond}
below, for  a precise definition):  there exist constants $d^+$ and $d^-$, such that:
\be \label{waves12} u(x)= d^+ e^{-\ i\nu x_1-\nu
x_2}+o(1),\,\,\, \text{as}\ \  x_1\rightarrow +\infty,
\ee
and
\be \label{waves122} u(x)= d^- e^{ i\nu
x_1-\nu x_2}+o(1),\,\,\, \text{as}\ \  x_1\rightarrow -\infty.
 \ee

\noindent Mathematical aspects of this problem
have been studied extensively, see e.g. \cite{J0}-\cite{KamMaz2}.
In particular, it is well known that the assumption of uniqueness of a solution implies solvability of the problem. We mention the  following  condition of uniqueness

\be
\label{Mazusl}
 x_1(x_1^2-x_2^2)n_1(x)+2x_1^2x_2n_2(x)\geq 0  ,\ \ \ x\in S,\bigskip
\ee
which was obtained  in  \cite{M1}.
The geometrical interpretation of this inequality  is discussed in \cite{hulme}.
Let us  consider the one-parametric family of circles passing through the origin, centered  at $(0,a)$, with $a >0$. The condition  \eqref{Mazusl} means that while  moving upwards from the origin along these circles we intersect $S$ at most once.
\bigskip
\subsection{Description of the main result}
The motivation for  the present paper is the fact that for the time being, there is no quantitative information on the dependence on the data and geometry of the domain for solutions to problem \eqref{w1}-\eqref{waves122}.
Our  goal here is to fill  this gap.

The main result will be obtained under the following
requirements, which together are  more stringent than \eqref{Mazusl}.

\bigskip
\noindent \textbf{Condition 1}. For all $x\in S$
\be
\label{conl4}
x_1n_1(x)\leq 0.\bigskip
\ee

\noindent
\textbf{Condition 2}. For a certain  $\ve \in (0,h] $, where $h$ is the distance from $\Gamma$ to $S$, there holds:
\be
\label{conl3}
 x_1(x_1^2-(x_2-\ve)^2)n_1(x)+2x_1^2(x_2-\ve)n_2(x)\geq 0  ,\ \ \ x\in S.\bigskip
\ee

\noindent The geometrical meaning of   Condition 1 is obvious,
while  Condition 2 has the same  sense as \eqref{Mazusl}, with the only difference that the circles pass through the point  $(0,\ve)$ instead of the origin.
We  show in Section 4 that Conditions 1 and 2 imply \eqref{Mazusl} which in turn implies uniqueness of the solution to problem \eqref{w1}-\eqref{waves122}.

Our principal result (Theorem \eqref{tm}) is a uniform in $\ve$ estimate of the sum
\bigskip
$$   \ve^2\left(|d^+|+|d^-|\right) +\ve^4\nu^2\left( \int_\Omega \frac{|\nabla u|^2 +\nu^2|u|^2}{1+\nu^2|x|^2}dx\right)^{1/2}
$$

\bigskip\noindent
by certain 
norms of the data $f, g_1, g_2$.

In order to give a complete formulation of the result, we need a number of  notations.
First, we assume that
\be \label{lhh}   B \subset \{ x : |\,x_1|<L,\,  h<x_2<H \}
\ee
for some positive numbers $L,h, H$ and maximum modulus of the curvature of $S$ is less than $\kappa$.

To measure the solution $u$ we introduce the norm
\bigskip
\be \label{norm1} \normu u= \left(\|\gamma_0 u\|_\Omega^2+\nu^{-2}\|\gamma_0 \nabla u\|_\Omega^2+\|\gamma_1 u\|_S^2+\|\gamma_2 u\|_\Gamma^2\right)^{1/2},
\ee

\bigskip
\noindent
where
 $\|w\|_\Xi$ stands for the $L_2$ norm of $w$ on the set $\Xi$. The weight functions $\gamma_0,\gamma_1$ and $\gamma_2$ are defined by
\bigskip
\be \label{gam0} \gamma_0^2=\nu^2(L^2\nu^2+1+\nu^{2}x_1^2+\nu^{2}x_2^2)^{-1},
\ee

\be \label{gam1} \gamma_1^2=(L+\nu^{-1}+H)^{-1},\ \gamma_2^2=\nu (1+\nu^{2}x_1^2)^{-1}.
\ee

\bigskip
\noindent
As for the right-hand side
\be \label{rhs} F=(f,g_1,g_2),
\ee
we measure it using the norm $\normf \cdot$, given by

\be \label{norm3} \normf F=
\ee

$$\left(\|\gamma_0^{-1} f\|_\Omega^2+\|\gamma_1^{-1} g_1\|_S^2\ +\|\gamma_1^{-1}x_1\partial_\sigma g_1 \|_S^2+\|\gamma_2^{-1} g_2\|_\Gamma^2+\nu^{-1}\|x_1\partial_{x_1}g_2\|_\Gamma^2\right)^{1/2},
$$

\noindent
where $\partial_\sigma$ is the first derivative with respect to the length of the contour $S$.

\bigskip
\noindent
Now we are in a position to formulate the main result.
\begin{theorem}  \label{tm}

Let  \eqref{conl4} and \eqref{conl3} hold. Further, let $\left(1+\nu|x|\right)f \in L_2(\Omega)$,    $\,\,g_1,\,  x_1\partial_\sigma g_1 \in L_2(S)$ and $\left(1+\nu|x_1|\right)g_2,\,  x_1\partial_{x_1}g_2  \in L_2(\Gamma)$. Then a unique solution $u$ of \eqref{w1}-\eqref{w3}, subject to the radiation conditions (see below \eqref{radcond}-\eqref{radcond17}),
 satisfies the estimate

\be \label{l100}\normu u \leq c\left(1+C\right)  \normf {F},\ \  |d^+|+|d^-|\leq c\left(1+C\right)^{1/2} \normf {F},
    \ee
where

\be C\leq h^{-2}\ve^{-2}(1+\kappa \tau)^4 (1+\nu  h)^6 \nu^3 \tau^{7},\ \ \ \ \tau=L+\nu^{-1}+H ,\bigskip
\ee
with   $F=(f,g_1,g_2)$ and an absolute positive constant $c$.
\end{theorem}

\bigskip

\subsection{Plan of the paper}
We demonstrate  in Section 2
that in order to estimate the scattering coefficients $d^+$ and $d^-$ and the norm of the solution $u$ it suffices to deal with solutions having the finite Dirichlet integral. In Section 3, a certain weighted $L
_2$ estimate for the tangential derivative on $S$ is obtained. Derivatives in the horizontal and vertical directions are estimated in Sections 4 and 5. One of  consequences of the  result  of Section 5 guarantees the unique solvability of problem \eqref{w1}-\eqref{waves122} for bodies sufficiently narrow in the horizontal direction.
Estimates for the solution and its boundary traces are given in Section 6. In the last Section 7 we collect the estimates previously obtained to complete the proof of Theorem \ref{tm}.

\bigskip

\section{\bf Estimate of scattering coefficients .}

\bigskip
We start with making precise the definition of the radiation conditions mentioned in Section 1.   In order to formulate \eqref{waves12}-\eqref{waves122} in more detail we introduce a cut off function $\chi\in C^{\infty}(\mathbb{R})$,  such that
$$\chi(t)=0\  \text{for}\   t<1\
 \text{and}\ \chi(t)=1 \ \text{for} \  t>3,  $$
 and $|\chi'|<1,  |\chi''|<1 $.
We say that $u$ satisfies radiation conditions if

 \be \label{radcond}u(x)=d^+U^+(x) +d^-U^-(x)
+ v(x),
 \ee
where
  \be \label{radcond7}U^+(x)=\chi(x_1\gamma_1^2)  e^{-\ i\nu x_1-\nu
x_2}, \ \ \ U^-(x)=d^-\chi(-x_1\gamma_1^2)  e^{\ i\nu x_1-\nu
x_2} \bigskip
 \ee
and $d^\pm$ are constants and $v$ is the remainder such that

\be\label{radcond17}\|\nabla v\|_\Omega+\|v \|_\Gamma<+\infty
\ee
 (compare with \eqref{waves122} ).
If
  the right-hand side of \eqref{w1}-\eqref{w3} is  smooth and compactly supported, then

 \be
 \label{newrad}   |\,v|=O(r^{-1})\ \ \ \text{and}\  \ \ |\nabla v|=O(r^{-2}), \ \ \text{as} \ \ r=\left(x_1^2+x_2^2\right)^{1/2}\rightarrow \infty,
\ee

\noindent
 see e.g. \cite{KMV} p. 46.

In this section, we verify that in order to estimate the norm of the solution $u$ it is enough to estimate the remainder $v$. Let us write down the boundary value problem for $v$:

\be
    -\Delta v=f_1 \ \ \text{in}\  \Omega, \label{w4}
\ee \be
    \partial_v v=g_1 \ \ \text{on}\   S, \label{w5}
\ee \be
    \partial_n v-\nu v=g_2\ \ \text{on}\    \Gamma, \label{w6} \ee
where
\be \label{w7} f_1= f +\Delta(d^+U^++d^-U^-).\ee

The function $v$ will be measured by the norm $\norm \cdot$ given by

\be \label{norm2}   \norm {v}^2= \|\gamma_0 v\|_\Omega^2+\| \nabla u\|_\Omega^2+\|\gamma_1 v\|_S^2+\nu\| v\|_\Gamma^2\ .\bigskip
\ee
Comparison of \eqref{norm2} and \eqref{norm1} leads to the inequality

\be \label{n5} \normu v\leq \norm v  \ .\bigskip
\ee

\begin{lemma}
\label{coef}  Let $f \in C_0^{\infty}(\Omega) $,   $g_1 \in C^{\infty}(S)$, $g_2 \in C^{\infty}_0 (\Gamma)$ and let a solution $v$ of \eqref{w4} - \eqref{w7}, be subject to the
 estimate

     \be \label {l1}\norm v \leq C_0 \normf {F_1}\, , \bigskip
    \ee
where $F_1=(f_1,g_1,g_2)$. Then for $F=(f,g_1 ,g_2)$ the estimate holds

\be\label{newoc77}  \norm v \leq C_1 \normf {F}\, , \bigskip
\ee

\noindent
and the solution $u$ of  problem \eqref{w1}-\eqref{w3}, \eqref{radcond}-\eqref{radcond17}  satisfies the  inequalities

\be \label{l10}\normu u \leq C_1 \normf {F}\, , \ \ \ \ \|d\|\leq C_2 \normf {F}\, , \bigskip
    \ee
where $d=(d^+,d^-)$ and $\|d\|^2=|d^+|^2+|d^-|^2$ and

\be \label{l12} C_1\leq c\left(1+\nu \gamma_1^{-2}C_0^2\right)     , \ \ \ \ \ C_2\leq c\left(1+\nu \gamma_1^{-2}C_0^2\right)^{1/2}  . \bigskip
\ee
\end{lemma}
\begin{proof} It follows from \eqref{radcond}, \eqref{n5}  and \eqref{l1} that

\be \label{l13} \normu u \leq \normu v +|d^+| \normu {U^+}+|d^-|\normu {U^-} \bigskip
\ee
$$
\leq \norm v +\|d\| A\leq C_0\normf {F_1}+\|d\| \mathcal{A},
$$
where
\be \mathcal{A}=\left(\normu {U^+}^2+\normu {U^-}^2\right)^{1/2}\leq 3.
\ee
We need to estimate $u$ by norm of $F$ not of $F_1$. To achieve this we first majorise  $F_1$ in terms of $F$ and $d$:

    \be\label {l15} \normf {F_1}\leq \normf F+ \normf {F_1-F}= \normf F+ \|\gamma_0^{-1} \Delta(d^+U^++d^-U^-)\|_\Omega
    \ee

    $$\leq \normf F + \|d\|\mathcal{B},
    $$
where

\be \label{l17} \mathcal{B}= \left(\|\gamma_0^{-1} \Delta U^-\|_\Omega^2+\|\gamma_0^{-1} \Delta U^+\|_\Omega^2\right)^{1/2}\leq 2^5\nu^{1/2} \gamma_1^{-1}. \bigskip
\ee
By \eqref{l13}
 and \eqref{l15} we conclude that

 \be\label{l19} \normu u\leq C_0\normf {F}+\|d\| (\mathcal{A}+C_0\mathcal{B}). \bigskip
    \ee
It remains to estimate $d$ by $F$.
Green's formula applied to $u$ and $\overline{u}$ implies:
    \be\label{Green}
    |d^+|^2+|d^-|^2=2\  \text{Im}\left(
    \int\limits_{\Omega}f\overline{u}dx+
    \int\limits_{S}g_1\overline{u}dS+\int\limits_{\Gamma}g_2\overline{u}dS
    \right)\, ,
    \ee

\noindent where we have used \eqref{newrad}
(see \cite{KMV}, p. 69, for an analogous formular in the case $f=0$,
$g_2=0$).  Representation \eqref{Green}  combined with  \eqref{norm1} and \eqref{norm3} yields

\be \label{l121} \|d\|^2\leq 2 \normu u\, \normf F.
\ee

\noindent
Inequalities \eqref{l19} and \eqref{l121} imply

\be  \label{l123}\normu u\leq 2\left(C_0+(\mathcal{A}+\mathcal{B}C_0)^2 \right)\normf F\,,
\ee
and
\be  \label{l124}\|d\|\leq 2\left(C_0+(\mathcal{A}+\mathcal{B}C_0)^2 \right)^{1/2}\normf F\,.
\ee
Finally estimates
\be
\mathcal{A}\leq  3 , \ \ \  \mathcal{B}\leq 2^5\nu^{1/2} \gamma_1^{-1},
\ee
together with \eqref{l123} and \eqref{l124}  lead to \eqref{l10} and \eqref{newoc77}. The proof of lemma is complete.
\end{proof}
\bigskip

\section{ \bf Estimate of tangential derivative on $S$}




 Let $T$ denote a constant and let $Z=(Z_1,Z_2)$ be a smooth real vector field in $\Omega$ with
 bounded derivatives
and $Z_2(x_1,0)=0$ for
all $x_1$.  Without loss of generality we suppose  that $v$, a solution of \eqref{w4}-\eqref{w7}, is real. The following identity is stated in \cite{M1}, see also \cite{KMV}, p.71:

    $$
        2\{(Z\cdot\nabla v +Tv)\Delta v\}=
        2\nabla\cdot\{(Z\cdot\nabla v+T v)\nabla v\}
    $$

    \be
        +(Q\nabla v)\cdot\nabla v-
        \nabla\cdot\left(|\nabla v|^2Z\right)\, .
        \label{ident}
    \ee

\noindent
Here $Q$ is a $2\times 2 $ matrix with the components $Q_{ij}=(\nabla\cdot Z-2T)\delta_{ij}-(\partial_i
Z_j+\partial_jZ_i),\ i,j=1,2$.  Let us  integrate \eqref{ident} over $\Omega$ and using \eqref{newrad}
integrate  by parts:

\be
    \int_\Omega 2\{(Z\cdot\nabla v +T v)\Delta v\}dx=2\int_{\partial \Omega}(Z\cdot\nabla v+T v)\partial_n
     v ds+
    \ee
    $$\int_\Omega (Q\nabla v)\cdot\nabla vdx-
        \int_{\partial \Omega}|\nabla v|^2(Z\cdot n)ds=
    $$
    $$
    2\int_{\Gamma}(Z\cdot\nabla v+T v)(\partial_{n}-\nu)
     v dx_1+2\nu\int_{\Gamma}(Z\cdot\nabla v+T v)
     v dx_1+
   $$
     $$
     2\int_{ S}(Z\cdot\nabla v+Tv)\partial_n
     v ds
     +\int_\Omega (Q\nabla v)\cdot\nabla vdx-
        \int_{S}|\nabla v|^2(Z\cdot n)ds.
    $$
 Since $Z_2=0$ on $\Gamma$ the right-hand side
 can be written as:

    \be \label{kkk}
    2\int_{\Gamma}(Z_1\partial_{x_1} v+T v)(\partial_{n}-\nu)
     v dx_1+2\nu\int_{\Gamma}(Z_1\partial_{x_1}v+T v)
     v dx_1+
\ee
     $$2\int_{ S}(Z\cdot\nabla v+T v)\partial_n
     v ds
     +\int_\Omega (Q\nabla v)\cdot\nabla vdx-
        \int_{S}|\nabla v|^2(Z\cdot n)ds.
    $$
 We have
 $$ Z\nabla v = (Z\cdot n )  v_n + (Z\cdot \sigma) v_\sigma,\bigskip
$$
where $\sigma$ is the unit tangential vector to $S$. Then  \eqref{kkk} can be expressed in the form

    $$
    2\int_{\Gamma}(Z_1\partial_{x_1} v+T v)(\partial_{x_2}-\nu)
     v dx_1+2\nu\int_{\Gamma}(Z_1\partial_{x_1}v+T v)
     v dx_1+
    $$

     $$
      2\int_{ S}((Z\cdot n)\partial_n v+ (Z\cdot \sigma)\partial_\sigma v+T v)\partial_n
     v ds
    +\int_\Omega (Q\nabla v)\cdot\nabla vdx-
        \int_{S}|\nabla v|^2(Z\cdot n)ds.
\bigskip    $$
Which can written as follows  integrating by parts
    $$
    2\int_{\Gamma}v(T-\frac{\partial}{\partial x_1}Z_1)(\partial_{x_2}-\nu)
     v dx_1+\nu\int_{\Gamma}(2T-\partial_{x_1} Z_1  )
     |v|^2 dx_1+
    $$
     $$
      2\int_{ S}((Z\cdot n)\partial_n v+ (Z\cdot \sigma)\partial_\sigma v+T v)\partial_n
     v ds
    +\int_\Omega (Q\nabla v)\cdot\nabla vdx-
        \int_{S}|\nabla v|^2(Z\cdot n)ds. \bigskip
    $$
Noting that
$$ |\nabla v|^2=v_\sigma^2+v_n^2,$$
we arrive at the identity.
$$\int_\Omega 2\{(Z\cdot\nabla v +T v)\Delta v\}dx=
$$

    $$
    2\int_{\Gamma}v(T-\frac{\partial}{\partial x_1}Z_1)(\partial_{x_2}-\nu)
     v dx_1+\nu\int_{\Gamma}(2T -\partial_{x_1} Z_1  )
     |v|^2 dx_1+
    $$

     $$
      \int_{ S}((Z\cdot n)\partial_n v+ 2(Z\cdot \sigma)\partial_\sigma v+2T v)\partial_n
     v ds
    +\int_\Omega (Q\nabla v)\cdot\nabla vdx-
        \int_{S}|\partial_\sigma v|^2(Z\cdot n)ds.\bigskip
    $$
Now using \eqref{w4}-\eqref{w7}, we obtain

\be\label{intidtt}
    \nu\int_{\Gamma}(\frac{\partial Z_1}{\partial x_1}  -2T )
     |v|^2 dx_1
    -\int_\Omega (Q\nabla v)\cdot\nabla vdx+
        \int_{S}|\partial_\sigma v|^2(Z\cdot
        n)ds=I(v,Z,T),
    \ee
where
$$
    I(v,Z,T)=2\int_{\Gamma}v(T-\frac{\partial}{\partial x_1}Z_1)g_2 dx_1+
$$
$$+\int_{ S}g_1\left((Z\cdot n)g_1+ 2(Z\cdot \sigma)\partial_\sigma
v+2T v\right) ds+\int_\Omega 2f_1(Z\cdot\nabla v +T v)dx.
$$

We shall use \eqref{intidtt} in the next Lemma to  estimate the tangential derivative of the solution $v$.

\begin{lemma} \label{oc} Let $f \in C_0^{\infty}(\Omega) $,   $g_1 \in C^{\infty}(S)$ and $g_2 \in C^{\infty}_0 (\Gamma)$.
Then a solution $v$ of \eqref{w4}-\eqref{w7} such that $\nabla v \in L_2(\Omega)$ and $v|_\Gamma\in L_2(\Gamma)$, satisfies the estimate

\be \int_{S}|\partial_\sigma v|^2(W\cdot
        n)ds\leq C_3\norm v \normf {F_1}+\normf {F_1} ^2,
\ee
where
\be F_1=(f_1,g_1,g_2) \, ,
\ee

 \be W(x_1,x_2)=\left(x_1\frac{x_1^2-x_2^2}{x_1^2+x_2^2},
 \frac{2x_1^2x_2}{x_1^2+x_2^2}\right ),\ee

\be \label{COO3} C_3=(3+2\left(\kappa L+10\right)),
\ee
and $\kappa$ is the maximum modulus of the curvature of $S$.
\end{lemma}
\begin{proof}
We put
 \be  Z=W,\ \   T=1/2.
 \ee
in \eqref{intidtt}.
 This is exactly the choice from \cite{M1} (see also \cite{KMV}, p.76.), where it was verified that   the quadratic form $(Q\nabla v)\cdot\nabla v$ is non--positive. Moreover the first term in \eqref{intidtt} vanishes.
 Then it follows from \eqref{intidtt} that,

\be\label{intid3}
        \int_{S}|\partial_\sigma v|^2(W\cdot
        n)ds\leq I(v,W,1/2),
    \ee
where
\be\label{z56}
    I(v,W,1/2)=-\int_{\Gamma}v(g_2+2x_1\partial_{x_1}g_2 )dx_1+
\ee
$$+\int_{ S}g_1\left((W\cdot n)g_1+ 2(W\cdot \sigma)\partial_\sigma
v+ v\right) ds+\int_\Omega f_1(2W\cdot\nabla v + v)dx.\bigskip$$
Let us estimate $I(v,W,1/2)$. For the last term in \eqref{z56} we have

\be \label{z57}\left| \int_\Omega f_1(2W\cdot\nabla v + v)dx \right|\leq 2\|f_1 W\|_\Omega \|\nabla v\|_\Omega+\|\gamma_0 ^{-1}f_1 \|_\Omega \|\gamma_0 v\|_\Omega
\ee

$$\leq2\|\gamma_0^{-1}f_1 \|_\Omega \|\nabla v\|_\Omega+\|\gamma_0 ^{-1}f_1 \|_\Omega \|\gamma_0 v\|_\Omega,
\bigskip$$
where we have used \eqref{gam0} and \eqref{newrad}.
For the first term on the right-hand side of \eqref{z56} we use \eqref{gam1} and obtain,

\be \label{z58}\left|\int_{\Gamma}v(g_2+2x_1\partial_{x_1}g_2 )dx_1\right|\leq \|g_2\|_\Gamma\|v\|_\Gamma+
2\|x_1\partial_{x_1}g_2\|_\Gamma\|v\|_\Gamma
\ee

$$\leq \nu^{1/2}\|v\|_\Gamma\|\gamma_2^{-1}g_2\|_\Gamma+
2\nu^{1/2}\|v\|_\Gamma \nu^{-1/2}\|x_1\partial_{x_1}g_2\|_\Gamma.\bigskip
$$
For the remaining term, on the right-hand side of \eqref{z56}, we have

\be \label{z59} \left|\int_{ S}g_1\left((W\cdot n)g_1+ 2(W\cdot \sigma)\partial_\sigma
v+ v\right) ds\right|\leq L\|g_1\|_S^2+
\bigskip\ee

$$\|\gamma_1^{-1}g_1\|_S\|\gamma_1 v\|_S+2\left|\int_{ S}v\partial_\sigma\left((W\cdot \sigma)g_1\right) ds\right|\leq
\bigskip$$

$$ \|\gamma_1^{-1}g_1\|_S^2+\|\gamma_1^{-1}g_1\|_S\|\gamma_1 v\|_S+2\|\gamma_1 v\|_S\|\gamma_1^{-1}(W\cdot \sigma)\partial_\sigma g_1 \|_S
\bigskip$$

$$ +2\left|\int_{ S}v g_1\partial_\sigma\left(W\cdot \sigma\right) ds\right|\bigskip
$$

$$\leq \|\gamma_1^{-1}g_1\|_S^2+\|\gamma_1^{-1}g_1\|_S\|\gamma_1 v\|_S+2\|\gamma_1 v\|_S\|\gamma_1^{-1}x_1\partial_\sigma g_1 \|_S
\bigskip$$

$$ +2\left(\kappa L+10\right)\|\gamma_1^{-1}g_1\|_S\|\gamma_1 v\|_S.\bigskip
$$
Combining \eqref{z56}-\eqref{z59} we complete the proof.
\end{proof}
\bigskip
\section{\bf Estimate  of the horizontal derivative.}
\bigskip
\bigskip
In this section we  estimate  $\partial_{x_1}v$ in the domain $\Omega$.
\begin{lemma} \label{oc77} Under the assumptions of Lemma \ref{oc} on $f, g_1, g_2$ and $v$,  the estimate holds

\be \|v_{x_1}\|_\Omega^2\leq C_4\norm v \normf {F_1}+2^{-1}\normf {F_1} ^2-2^{-1}\int_{S}|\partial_\sigma v|^2x_1n_1(x)ds.
\ee
where
\be C_4= \sqrt{2}+\kappa L\leq 2^{-1}C_3.
\ee

\end{lemma}
\begin{proof}
We put in \eqref{intidtt} $T=1/2, Z=V$, where
 $$V(x)=(x_1,0).$$
 Then
 \be 2\|v_{x_1}\|_\Omega^2+\int_Sx_1n_1(x) |\partial_\sigma v|^2ds= I(v,V,1/2),
 \ee
where
\be\label{z80}
    I(v,V,1/2)=-\int_{\Gamma}v(g_2+2x_1\partial_{x_1}g_2 )dx_1+
\ee
$$+\int_{ S}g_1\left((V\cdot n)g_1+ 2(V\cdot \sigma)\partial_\sigma
v+ v\right) ds+\int_\Omega f_1(2V\cdot\nabla v + v)dx.$$
Next we estimate   $I(v,V,1/2)$. For the last term in \eqref{z80} we have

\be \label{z81}\left| \int_\Omega f_1(2V\cdot\nabla v + v)dx \right|\leq 2\|f_1 V\|_\Omega \|\nabla v\|_\Omega+\|\gamma_0 ^{-1}f_1 \|_\Omega \|\gamma_0 v\|_\Omega
\bigskip\ee
$$\leq2\|\gamma_0^{-1}f_1 \|_\Omega \|\nabla v\|_\Omega+\|\gamma_0 ^{-1}f_1 \|_\Omega \|\gamma_0 v\|_\Omega,
\bigskip$$
where we have used \eqref{gam0}.
In order to estimate the first term on the right-hand side of \eqref{z80} we use \eqref{gam1} and obtain,
\be \label{z82}\left|\int_{\Gamma}v(g_2+2x_1\partial_{x_1}g_2 )dx_1\right|\leq \|g_2\|_\Gamma\|v\|_\Gamma+
2\|x_1\partial_{x_1}g_2\|_\Gamma\|v\|_\Gamma
\ee

$$\leq \nu^{1/2}\|v\|_\Gamma\|\gamma_2^{-1}g_2\|_\Gamma+
2\nu^{1/2}\|v\|_\Gamma \nu^{-1/2}\|x_1\partial_{x_1}g_2\|_\Gamma.
$$

\bigskip
\noindent
Finally for the second integral on the right-hand side of \eqref{z80} we have

\be \label{z83} \left|\int_{ S}g_1\left((V\cdot n)g_1+ 2(V\cdot \sigma)\partial_\sigma
v+ v\right) ds\right|\leq L\|g_1\|_S^2+\bigskip
\ee
$$\|\gamma_1^{-1}g_1\|_S\|\gamma_1 v\|_S+2\left|\int_{ S}v\partial_\sigma\left((V\cdot \sigma)g_1\right) ds\right|\leq\bigskip
$$
$$ \|\gamma_1^{-1}g_1\|_S^2+\|\gamma_1^{-1}g_1\|_S\|\gamma_1 v\|_S+2\|\gamma_1 v\|_S\|\gamma_1^{-1}(V\cdot \sigma)\partial_\sigma g_1 \|_S\bigskip
$$
$$ +2\left|\int_{ S}v g_1\partial_\sigma\left(V\cdot \sigma\right) ds\right|\bigskip
$$
$$\leq \|\gamma_1^{-1}g_1\|_S^2+\|\gamma_1^{-1}g_1\|_S\|\gamma_1 v\|_S+2\|\gamma_1 v\|_S\|\gamma_1^{-1}x_1\partial_\sigma g_1 \|_S\bigskip
$$
$$ +2\left(\kappa L+1\right)\|\gamma_1^{-1}g_1\|_S\|\gamma_1 v\|_S.\bigskip
$$
Combining \eqref{z80}-\eqref{z82} we arrive at

\be 2\|v_{x_1}\|_\Omega^2+\int_{S}|\partial_\sigma v|^2x_1n_1(x)ds\leq (2\sqrt{2}+2\kappa L)\norm v \normf {F_1}+\normf {F_1} ^2.
\ee
\end{proof}

\begin{lemma} \label{oc7777} Under Conditions 1 and 2 (see \eqref{conl4},\eqref{conl3} ) we have
 the estimate:
\be -x_1n_1\leq \frac{H}{\ve} W(x)\cdot n(x),  \ x\in S. \label{oc777}
\ee
\end{lemma}
\begin{proof}
Let us notice that if for some point $x\in S$ we have $x_2=h$, then inequality \eqref{oc777} is obvious since then $x_1n_1(x)=0$, (see \eqref{lhh}) and $W(x)\cdot n(x)\geq 0$ by the same reason.
As a result we can assume that
\be \label{oc776} x_2>h\geq\ve.
\ee
Condition \eqref{conl3} implies,

$$ 0\leq \|x\|^{-2}\left(x_1(x_1^2-(x_2-\ve)^2)n_1+2x_1^2(x_2-\ve)n_2\right)$$
\be \label{oc778}  =W(x)\cdot n(x)+\ve\|x\|^{-2}\left( 2x_1x_2n_1-2x_1^2n_2\right)-\ve^2\|x\|^{-2}x_1n_1.
\ee
We see  that the term in brackets in \eqref{oc778} can be written as,

\be     \|x\|^{-2}\left( 2x_1x_2n_1-2x_1^2n_2\right)=\|x\|^{-2}x_2^{-1}\left( x_1n_1\|x\|^2-x_1(x_1^2-x_2^2)n_1-2x_1^2x_2n_2\right)
\ee

$$ =    x_2^{-1}x_1n_1 -x_2^{-1} W(x)\cdot n(x) .
$$
Then   \eqref{oc778} implies:
\be 0\leq W(x)\cdot n(x)\left( 1-\frac{\ve}{x_2}\right)+\frac{\ve}{x_2}x_1n_1\left( 1-\frac{\ve x_2 }{\|x\|^2}\right)\leq
\ee
$$W(x)\cdot n(x)\left( 1-\frac{\ve}{x_2}\right)+\frac{\ve}{x_2}x_1n_1\left( 1-\frac{\ve  }{x_2}\right),
$$
where we have used Condition 1. By \eqref{lhh} and \eqref{oc776} we arrive at \eqref{oc777}.
\end{proof}
Combining Lemma \ref{oc}, Lemma \ref{oc77} and Lemma \ref{oc7777} we arrive at the following assertion.

\begin{theorem}  \label{T4} Under Conditions 1 and 2 (see \eqref{conl4}, \eqref{conl3} )  and the assumptions of Lemma \ref{oc} on $f, g_1, g_2$ and $v$ we have
 the estimate:

\be  \|v_{x_1}\|_\Omega^2\leq C_5\norm v\, \normf {F_1}+C_6 \normf {F_1} ^2,
\ee
where
\be\label{COO5} C_5\leq C_4+\frac{HC_3}{2\ve}\leq C_3\frac{H}{\ve},
\ee
and
\be \label{COO6} C_6\leq \frac{H}{2\ve}+\frac{1}{2}\leq \frac{H}{\ve}.
\ee
\end{theorem}

\bigskip

\section{\bf {Estimate of the vertical derivative} }

In this section we do not use Conditions 1 and 2. We only assume that,
\be   \label{simas} B \subset \{ x : |\,x_1|<L,\,  h<x_2 \}
\ee
for some positive numbers $L,h$.

Let us wright the   identity
\be \label{energy id}
\int_\Omega f_1v dx+\int_S g_1v dS+\int_\Gamma g_2vdx_1=\int_\Omega |\nabla v|^2dx-\nu \int_\Gamma v^2dx_1.
\ee
and put
\be
\label{III} J:=\int_\Omega f_1v dx+\int_S g_1v dS+\int_\Gamma g_2vdx_1.
\ee

The next lemma will be used in the proof of Theorem \ref{tm}.  But it is also of independent interest.
\begin{lemma}\label{T5} Under conditions \eqref{simas}   and the assumptions of Lemma \ref{oc} on $f, g_1, g_2$ and $v$, the estimate holds

\be \label{l1form}\|\nabla v\|^2_\Omega \leq 4J+18 \left(\|r f_1\|^2_\Omega+\nu \|r g_2\|^2_\Gamma\right)+(C_7-3)\|v_{x_1}\|_\Omega^2 ,
\ee
where
\be  C_7\leq   72 \frac{\nu L^2}{h} (1+\nu h)^3.
\ee

\end{lemma}

\begin{proof}

We divide the right-hand side of \eqref{energy id} in three parts:

\be\label{sep}
\int_\Omega |\nabla v|^2dx-\nu \int_\Gamma v^2dx_1=I_1+I_2+I_3,
\ee
where
\be   I_j=\int_{\Omega_j} |\nabla v|^2dx-\nu \int_{\Gamma_j} v^2dx_1,\ \ \  j=1,2,3\, ,
\ee
with

$$   \Omega_1=\Omega \cap \left\{x_1<-L\right\}, \,\   \Omega_2=\Omega\cap\{-L <x_1<L\},
$$

$$\Omega_3=\Omega\cap\{ L<x_1 \}
$$
and

$$ \Gamma_1=\Gamma \cap \left\{x_1<-L\right\}, \,\   \Gamma_2=\Gamma\cap\{-L <x_1<L\},$$

$$ \Gamma_3=\Gamma\cap\{ L<x_1 \} .
$$
Let us treat $I_1$.
We consider the orthogonal decomposition of  $v$  in $\Omega_1$, see \cite{J0}:

\be  \label{rep} v(x)=\Theta(x)+a(x_1)\Phi(x_2),
\ee
where
\be \label{not}    \Phi(x_2)=\sqrt{2\nu}\,e^{-\nu x_2},\,\,\,\  a(x_1)=\int_{0}^{+\infty}v(x_1,x_2)\Phi(x_2)dx_2.
\ee
Substituting \eqref{rep} into the definition of $I_1$ we get

\be  \label{i1}
I_1=\int_{\Omega_1}  v_{x_1}^2dx+\frac{1}{4}\int_{\Omega_1}  v_{x_2}^2dx+\frac{3}{4}\int_{\Omega_1} \Theta_{x_2}^2dx+\frac{3}{4}\nu^2\int_{\Omega_1} a^2(x_1) \Phi^2(x_2)dx\ee

$$-\nu \int_{\Gamma_1} \Theta^2dx_1-\nu \int_{\Gamma_1} a^2(x_1)\Phi^2(0)dx_1=
$$

$$\int_{\Omega_1}  v_{x_1}^2dx+\frac{1}{4}\int_{\Omega_1}  v_{x_2}^2dx+\left(\frac{3}{4}\int_{\Omega_1} \Theta_{x_2}^2dx-\nu \int_{\Gamma_1} \Theta^2dx_1\right)-\nu^2 \frac{5}{4}\int_{\Gamma_1} a^2(x_1)dx_1.
$$

\bigskip
\noindent
To estimate the third term in the right-hand side of \eqref{i1} we use the
F. John's estimate, see \cite{J0},

\be  \label{sm}  \int_{\Omega_1} \Theta_{x_2}^2dx/2\geq \nu \int_{\Gamma_1} \Theta^2dx_1,
\ee
which follows from orthogonality condition

\be     \int_{0}^{+\infty}\Theta(x_1,x_2)e^{-\nu x_2}dx_2=0, \,\, \,\, x_1\in(\infty,-L),\ee
see \eqref{rep} and \eqref{not}. Then

\be \label{3333}I_1\geq \int_{\Omega_1}  v_{x_1}^2dx+\frac{1}{4}\int_{\Omega_1}  v_{x_2}^2dx+\frac{1}{4}\int_{\Omega_1} \Theta_{x_2}^2dx- \frac{5\nu^2}{4}\int_{\Gamma_1} a^2(x_1)dx_1.
\ee
We wish to check that  the third term on the right-hand side of \eqref{3333} controls the $L_2$ norm of $\Theta$ in the horisontal semi-strip
$$\Omega_1':=(-\infty,-L)\times(0,h).$$
In fact, for any $w\in H^1(0,N)$ and positive $\alpha$ and $N$
we have  the obvious inequality

\be \label{firstest}
M_N(\alpha)\int_0^N w^2dt\leq \alpha w^2(0)+\int_0^N |w'|^2dt,
\ee
where
$$ M_N(\alpha)\geq \frac{\pi^2\alpha}{N(4\alpha N+\pi^2)}\geq\frac{\alpha}{N(\alpha N+1)}.$$
Applying \eqref{firstest} (with $\alpha=\nu $ and $N=h$) to the  third term on the right-hand side of \eqref{3333} and
using  \eqref{sm} we derive,
\be \label{zzz1} M_h(\nu)\|\Theta\|^2_{\Omega_1'}\leq \nu \|\Theta\|^2_{\Gamma_1}+\|\Theta_{x_2}\|_{\Omega_1'}^2\leq 2\|\Theta_{x_2}\|_{\Omega_1}^2.
\ee
On the other hand, we have

\be \label{zzz2}\|v\|_{\Omega_1'}^2\leq 2\|\Theta\|_{\Omega_1'}^2+ 2\|a\Phi\|_{\Omega_1'}^2= \ee

$$2\|\Theta\|_{\Omega_1'}^2+2\|a\|_{\Gamma_1}^2\left (1-e^{-2\nu h}\right)\leq 2\|\Theta\|_{\Omega_1'}^2 +4\nu h\|a\|_{\Gamma_1}^2.
$$

\bigskip
\noindent
Combining \eqref{zzz1} and \eqref{zzz2} we obtain,

\be\label{zzz3}
\frac{M_h(\nu)}{4}\|v\|_{\Omega_1'}^2- M_h(\nu)\nu h \|a\|_{\Gamma_1}^2\leq \|\Theta_{x_2}\|_{\Omega_1}^2.
\ee
By \eqref{zzz3} and \eqref{3333}  we have

\be \label{33}I_1\geq \|v_{x_1}\|_{\Omega_1}^2+\frac{1}{4}\|v_{x_2}\|_{\Omega_1}^2+\frac{M_h(\nu)}{16}\| v\|^2_{\Omega_1'}-\left(\frac{M_h(\nu)\nu h}{4}+ \frac{5\nu^2}{4}\right) \|a\|_{\Gamma_1}^2
\ee

$$\geq \|v_{x_1}\|_{\Omega_1}^2+\frac{1}{4}\|v_{x_2}\|_{\Omega_1}^2+\frac{M_h(\nu)}{16}\| v\|^2_{\Omega_1'}-\frac{3\nu^2}{2} \|a\|_{\Gamma_1}^2.
$$

\bigskip
\noindent
In order to estimate the last term on  the right-hand side of \eqref{33}, we need to return  to equation  \eqref{w4} and  boundary conditions
\eqref{w6}. It follows from the orthogonal decomposition \eqref{rep} that the function $a$ satisfies the equation

\be\label{eq aa}
\partial_{x_1}^2 a(x_1)+\nu^2 a(x_1)= f_2(x_1),   \,\, \text{for a.e.}\,\, x_1\in(\infty,-L),
\ee
where

\be   f_2(x_1)=\int_{0}^{+\infty} f_1(x_1,x_2)\Phi(x_2)dx_2-\sqrt{2\nu} g_2(x_1).
\ee
Multiplying \eqref{eq aa} by $2x_1 \partial_{x_1} a(x_1) $, integrating over  $(\infty,-L)$ and integrating by parts, we obtain
\be 2\int_{-\infty}^{-L} f_2x_1 a_{x_1}dx_1=
\ee
$$-L\left(\left(a_{x_1}(-L)\right)^2+\nu^2 \left(a(-L)\right)^2\right)-\int_{-\infty}^{-L} a_{x_1}^2dx_1-\nu^2\int_{-\infty}^{-L} a^2dx_1
$$
Therefore,

\be  \label{gggg}  \nu^2\int_{-\infty}^0 a^2dx_1\leq \int_{-\infty}^{-L} x_1^2f_2^2dx_1
 \leq 3\|x_1f_1\|_{\Omega_1}^2+3\nu\|x_1 g_2\|_{\Gamma_1}^2.
\ee
By \eqref{gggg} and \eqref{33}

\be\label{i1fin}I_1\geq \|v_{x_1}\|_{\Omega_1}^2+\frac{1}{4}\|v_{x_2}\|_{\Omega_1}^2+\frac{M_h(\nu)}{16}\| v\|^2_{\Omega_1'}-
\frac{9}{2}\left( \|x_1f_1\|_{\Omega_1}^2+\nu\|x_1 g_2\|_{\Gamma_1}^2 \right).
\ee
The estimate of the same type holds for $I_3$:

\be\label{i3fin}I_3\geq \|v_{x_1}\|_{\Omega_3}^2+\frac{1}{4}\|v_{x_2}\|_{\Omega_3}^2+\frac{M_h(\nu)}{16}\| v\|^2_{\Omega_3'}-
\frac{9}{2}\left( \|x_1f_1\|_{\Omega_3}^2+\nu\|x_1 g_2\|_{\Gamma_3}^2 \right),
\ee
where $$ \Omega_3'=(L,+\infty)\times (0,h).$$
Let us treat $I_2$.
We have the following auxiliary estimate,

\be
-\frac{\alpha(1+\alpha N)}{N}\int_0^N w^2dt\leq -\alpha w^2(0)+\int_0^N |w'|^2dt,
\ee
Applying this  inequality to $I_2$ we arrive at:

\be \label{i2fin}I_2\geq \|v_{x_1}\|^2_{\Omega_2}+\frac{1}{4} \|v_{x_2}\|^2_{\Omega_2} +\frac{3}{4} \|v_{x_2}\|^2_{\Omega_2'}-\nu \|v\|^2_{\Gamma_2}
\ee

$$ \geq\|v_{x_1}\|^2_{\Omega_2}+\frac{1}{4} \|v_{x_2}\|^2_{\Omega_2} -\frac{\nu(3+4\nu h)}{3h} \|v\|^2_{\Omega_2'} ,
$$
where

 $$ \Omega_2'=(-L,L)\times (0,h).$$
Combining \eqref{i1fin}, \eqref{i3fin} and \eqref{i2fin} we obtain
\be I\geq \|v_{x_1}\|^2_{\Omega}+\frac{1}{4}\|v_{x_2}\|^2_{\Omega}+R \left(\|v\|_{\Omega_1'}^2+\|v\|_{\Omega_3'}^2\right)
\ee

$$-K\|v\|^2_{\Omega_2'}-\frac{9}{2}\left( \|x_1f\|_{\Omega}^2+\nu\|x_1 g_2\|_{\Gamma}^2 \right),
$$
where
\be R\geq\frac{M_h(\nu)}{16} , \,\,\,  K\leq \frac{\nu(1+\nu h)}{h}.
\ee
We shall use elementary inequality

\be  \label{aux5} M_L[\alpha]\int_0^L w^2dt\leq \int_0^{+\infty} |w'|^2dt+ \alpha^2\int_L^{+\infty} w^2dt.
\ee
Applying \eqref{aux5} twice we get

\be \label{ppp} I\geq \|v_{x_1}\|^2_{\Omega}+\frac{1}{4}\|v_{x_2}\|^2_{\Omega}-\frac{9}{2}\left( \|x_1f_1\|_{\Omega}^2+\nu\|x_1 g_2\|_{\Gamma}^2 \right)
-A\|v_{x_1}\|^2_{\Omega'}
\ee
$$+\left(A M_L\left[R^{1/2}A^{-1/2}\right]-K\right)\|v\|^2_{\Omega_2'},
$$
with an arbitrary positive $A$. It remains to choose $A$ so that the last term in \eqref{ppp} is positive, i.e
\be \label{kva} AR^{1/2}-A^{1/2}KL-KL^2R^{1/2}>0.
\ee
Solving this  inequality, quadratic with respect to $A^{1/2}$, we see that \eqref{kva} follows from
\be \label{ttt} A\geq L^2 \frac{(K+\sqrt{K^2+4KR}\ )^2}{4R}.
\ee
Since
$$L^2 \frac{(K+\sqrt{K^2+4KR}\ )^2}{4R}\leq L^2K\left(KR^{-1}+2\right)\leq L^2\frac{\nu(1+\nu h)}{h}\left (16(1+\nu h)^2+2\right),$$
we conclude that the last term in \eqref{ppp} is positive if
$$A\geq 18 \frac{\nu L^2}{h} (1+\nu h)^3.
$$
Now, \eqref{l1form} follows from \eqref{ppp} and \eqref{ttt}.
\end{proof}

We state a  uniqueness result for problem \eqref{w1}-\eqref{waves122} which follows directly from \eqref{l1form}.
\begin{corollary} \label{rem1} If
\be\label{newun} 24 \frac{\nu L^2}{h} (1+\nu h)^3<1,
\ee
then problem \eqref{w1}-\eqref{waves122} is uniquely solvable.

Note that here we did not use Conditions 1 and 2. In particular the unique solvability holds either the body $B$ is sufficiently narrow in the horisontal direction or $\nu$ is small.
\end{corollary}
\begin{rem} \label{rem2} Using notations from Section 1 we write  \eqref{l1form} as

\be\label{newoc} \|\nabla v\|^2_\Omega \leq 4\normf {F_1} \norm v +18\normf {F_1} ^2 +C_7\|v_{x_1}\|_\Omega^2,
\ee
where
\be  \label{COO7} C_7=  72 \frac{\nu L^2}{h} (1+\nu h)^3.
\ee
\end{rem}

\section{\bf Estimates of the function and its traces }

In this section we estimate the solution $v$ and its boundary traces in weighted $L_2$ spaces.
\begin{lemma}\label{L6} Let  \eqref{lhh} and the assumptions of Lemma \ref{oc} on $f, g_1, g_2$ and $v$ hold. Then one has the estimates

\be \label{l6form1}
\nu \| v\|^2_\Gamma \leq 3\normf {F_1} \norm v +18\normf {F_1}^2 +C_7\|v_{x_1}\|_\Omega^2,
\ee

\be \label{l6form2}
\| \gamma_0 v\|^2_\Omega \leq 16 \left(\| \nabla v\|^2_\Omega +\frac{\nu}{2} \| v\|^2_\Gamma\right),
\ee

\be \label{l6form3}
\|\gamma_1 v\|^2_S \leq  \| \nabla v\|^2_\Omega +\nu \| v\|^2_\Gamma  +C_8\| \gamma_0 v\|^2_\Omega ,
\ee
where
\be  C_7=   72 \frac{\nu L^2}{h} (1+\nu h)^3,
\ee
and
\be \label{COO8} C_8=24+18\kappa \gamma_1^{-2}.
\ee
\end{lemma}
\begin{proof}
1. Estimate \eqref{l6form1} follows  from \eqref{l1form} and \eqref{energy id}.

2. We majorise $L_2$ norm of $v$ in $\Omega$. First we have auxiliary Hardy type inequality
\be\label{hardy}
\frac{\nu}{2} w^2(0)+\int_0^{+\infty}|w'|^2dt\geq  \frac{1}{4}\int_0^{+\infty}\frac{\nu^2}{(\nu t+1)^2}w^2dt.
\ee
Hence
\be \label{omeg13}\|\nabla v\|^2_\Omega +\frac{\nu}{2} \|v\|^2_{\Gamma}\geq \int_{\Omega_1\cup\ \Omega_3} \mu^2(x_2)v^2 dx+\| v_{x_1}\|^2_{\Omega\setminus \Omega_2}+\|\nabla v\|^2_{\Omega_2},
\ee
where
\be \mu(x_2)=\frac{\nu}{2(1+\nu x_2)}. \label{mu77}
\ee
 In order to estimate $v$ in $\Omega_2$ we  use \eqref{omeg13} and \eqref{aux5} and obtain

\be \label{omeg2}\|\nabla v\|^2_\Omega +\frac{\nu}{2} \|v\|^2_{\Gamma}\geq \int_{\Omega_2} M_{L}[\mu(x_2)]\,v^2 dx.
\ee
Adding \eqref{omeg13} and \eqref{omeg2}, we arrive at

\be  \label{omeg123}2\|\nabla v\|^2_\Omega +\nu \|v\|^2_{\Gamma}\geq \frac{1}{8} \int_{\Omega_2} \rho^2(x_2)\,v^2 dx,
\ee
where
$$\rho^2(x_2)=8\, \textrm{min}\,\left\{\mu^2(x_2), M_{L}[\mu(x_2)]\right\}\geq \frac{8 \mu^2}{2L^2\mu^2 +1}=\frac{8}{2L^2+4\left(\nu^{-1}+x_2\right)^2}\geq \gamma_0^2,$$
from which \eqref{l6form2} follows.

3. Let us  introduce curvilinear coordinates in a neighborhood of $S$: $(s,\rho)$, where $\rho$ is the distance to $S$ and $s$ is the coordinate on $S$.  We choose a smooth cut off function $\theta$ such that
$$\theta(0)=1 \ \ \text{for}\ \  t=0\ \  \text{and}\ \ \theta(t)=0 \ \  \text{for}\ \  t>2.
$$
and $|\theta|\leq 1$, $|\theta\,'|\leq 1$. Then   the function $\theta\left(4(\kappa+\gamma_1^2)\rho\right)\nabla \rho$ is well defined in $\Omega$.

Consider the integral
\be T=\int_\Omega \theta\left(4(\kappa+\gamma_1^2)\rho\right)\nabla \rho \nabla v^2 dx.
\ee
Obviously, we have

\be \label{true1} |T|\leq 2\|v\|_{\tilde{\Omega}}\|\nabla v\|_{\tilde{\Omega}},
\ee
where
\be
\label{omh}
\tilde{\Omega}=\Omega\cap \{ x : 0<\rho <2^{-1}(\kappa+\gamma_1^2)^{-1}\}.
\ee
On the other hand, integrating by parts,  we get

\be \label{TTTT} T=-\int_\Omega v^2\theta\left(4(\kappa+\gamma_1^2)\rho\right)\Delta \rho  dx-\int_\Omega v^2\nabla \theta\left(4(\kappa+\gamma_1^2)\rho\right)\cdot \nabla \rho  dx
\ee

$$+\int_{\partial\Omega} v^2\theta\left(4(\kappa+\gamma_1^2)\rho\right)\partial_n \rho  ds=
$$

$$ -\int_\Omega v^2\theta\left(4(\kappa+\gamma_1^2)\rho\right)\Delta \rho dx-4(\kappa+\gamma_1^2)\int_\Omega v^2 \theta\,'\left(4(\kappa+\gamma_1^2)\rho\right) \nabla \rho \cdot \nabla \rho dx-\int_S v^2  ds
$$

$$+\int_\Gamma v^2\theta\left(4(\kappa+\gamma_1^2)\rho\right)\partial_n \rho  ds.
$$

\noindent
Let $k$ be the curvature of $S$. Noting that

$$|\Delta \rho| = |k\left(1+k\rho\right)^{-1} |\leq 2\kappa, \ \ \text{in}\  \tilde\Omega ,
$$
and
$$ \nabla \rho \cdot \nabla \rho = 1,
$$
we conclude from \eqref{TTTT} that

\be \label{true2} \int_S v^2  ds \leq |T|+\|v\|_\Gamma^2+ (6\kappa+4\gamma_1^2)\|v\|^2_{\tilde{\Omega}}.
\ee
Combining \eqref{true1} and \eqref{true2} we arrive at

\be  \label{z56t} \|\gamma_1 v\|_S^2\leq \|\nabla v\|_{\Omega}^2 + \nu  \|v\|^2_{\Gamma}+6(\gamma_1^{4}+\gamma_1^2 \kappa)\|v\|^2_{\tilde{\Omega}}.
\ee

\noindent It remains to estimate the last term in \eqref{z56t}. Using \noindent\eqref{omh} and \eqref{gam0}, we obtain

\be \label{bgt1}\|v\|^2_{\tilde{\Omega}}\leq \left(\max_{x\in {\tilde{\Omega}}}\gamma_0^{-2} \right)\|\gamma_0 v\|^2_{\Omega}\leq (3\gamma_1^{-4}+(\kappa+\gamma_1^2)^{-2})\|\gamma_0 v\|^2_{\Omega}.
\ee

Inequality \eqref{l6form3} follows from \eqref{bgt1} and \eqref{z56}. This concludes the proof.
\end{proof}

\bigskip

\section{\bf Proof of the main result}

\bigskip

\begin{theorem}  \label{T10} Let $f \in C_0^{\infty}(\Omega) $,   $g_1 \in C^{\infty}(S)$ and $g_2 \in C^{\infty}_0 (\Gamma)$.
Then a unique solution $v$ of \eqref{w4}-\eqref{w7} such that $\nabla v \in L_2(\Omega)$ and $v|_\Gamma\in L_2(\Gamma)$, satisfies the estimate

\be \norm v  \leq C_0 \normf {F_1},
\ee
where
\be \label{COO} C_0\leq 2^{30}H\ve^{-1}(1+\kappa L)(1+\kappa \gamma_1^{-2})\left(1+\nu L^2h^{-1}(1+\nu h)^3\right)
\ee

$$ \leq c\,h^{-1}\ve^{-1}(1+\kappa \tau)^2 (1+\nu  h)^3 \nu \tau^{3},
$$
and
$$  \tau=L+H+\nu^{-1}.
$$
\end{theorem}
\begin{proof}
According to Lemma \ref{L6} and Remark \ref{rem2} the  inequality holds

\be\label{tre}
\norm v ^2\leq   36(1+C_8)\left(4\norm v \normf {F_1}+18\normf {F_1}^2+C_7\|v_{x_1}\|_\Omega\right).
\ee

\bigskip
\noindent
The right-hand side of \eqref{tre} is dominated  by

$$   36(1+C_8)\left((4+C_5C_7)\norm v \normf {F_1}+(18+C_6C_7)\normf {F_1}^2\right),
$$

\bigskip
\noindent
see Theorem \ref{T4},    which is not greater than

\be \label{tre1}
  2^{11}C_8\left((1+C_7C_5)\norm v \normf {F_1}+(1+C_7C_6)\normf {F_1}^2\right).
\ee

\bigskip
\noindent
Using \eqref{COO5} and \eqref{COO6}, we majorise \eqref{tre1} by

$$
  2^{11}C_8\left((1+C_7C_3H\ve^{-1})\norm v \normf {F_1}+(1+C_7H\ve^{-1})\normf {F_1}^2\right).
$$

\bigskip
\noindent
Noting that $C_3 H\ve^{-1}\geq1$, because \eqref{COO3}, we arrive at the inequality

\be \label{tre2}
   \norm v ^2\leq 2^{11}H\ve^{-1}C_8C_3(1+C_7)\left(\norm v \normf {F_1}+\normf {F_1}^2\right).
\ee

\bigskip
\noindent
Since $H\ve^{-1}C_8C_3(1+C_7)\geq 1$, we conclude

\be\label{COOO}  \norm v \leq 2^{12}H\ve^{-1}C_8C_3(1+C_7)\normf {F_1}.
\ee

\noindent It remains  to apply  \eqref{COO3}, \eqref{COO7} and \eqref{COO8} to  \eqref{COOO} to conclude the proof.
\end{proof}

Now, we can finish  the proof of Theorem \ref{tm}.

\begin{proof} Let $f \in C_0^{\infty}(\Omega) $,   $g_1 \in C^{\infty}(S)$ and $g_2 \in C^{\infty}_0 (\Gamma)$. Then
\eqref{l100}  and \eqref{newoc77} follow from Theorem \ref{T10}
and Lemma \ref{coef}.  We conclude the proof by approximating the right-hand side of \eqref{w1}-\eqref{w3} by $C_0^\infty$ functions  in the norm $\normf \cdot$ and passing to the limit in \eqref{radcond}, \eqref{l100} and \eqref{newoc77}.
\end{proof}


\end{document}